\documentclass[12pt,a4paper,reqno]{amsart}
\usepackage{amssymb}
\usepackage[T1]{fontenc}
\usepackage{mathptmx}
\usepackage[libertine]{newtxmath}
\usepackage{mathrsfs}
\usepackage{tikz-cd}
\usepackage[headings]{fullpage}
\usepackage{paralist}
\usepackage{xspace}
\usepackage{fonttable}
\usepackage{amsmath}
\usepackage{color}
\usepackage[normalem]{ulem}

\usepackage[colorlinks=true,linkcolor=blue, citecolor=blue, urlcolor=blue, pagebackref=true]{hyperref}

\usepackage{array, booktabs}
\usepackage{makecell}
\setcellgapes{2.5pt}
\theoremstyle{plain}
\newtheorem{theorem}[equation]{Theorem}
\newtheorem{lemma}[equation]{Lemma}

\newtheorem{proposition}[equation]{Proposition}

\theoremstyle{definition}

\newtheorem{remark}[equation]{Remark} 

\newtheorem{example}[equation]{Example} 

\newtheorem{definition}[equation]{Definition}

\newtheorem{notation}[equation]{Notation}

\newtheorem{discussion}[equation]{Discussion}

\newtheorem{observation}[equation]{Observation}

\newtheorem{construction}[equation]{Construction}

\newcommand{\er}{^e\!R}
\DeclareMathOperator{\rank}{rank}
\newcommand{\frakm}{{\mathfrak m}}

\newcommand{\calP}{\mathcal P}
\DeclareMathOperator{\ehk}{e_{HK}}
\newcommand{\naturals}{\mathbb{N}}
\newcommand{\ints}{\mathbb{Z}}
\newcommand{\bbk}{\mathbb{k}}
\def\too{\longrightarrow}

\DeclareMathOperator{\charact}{char}
\DeclareMathOperator{\projective}{\mathbb{P}}
\DeclareMathOperator{\Proj}{Proj}
\DeclareMathOperator{\Tor}{Tor}

\begin{document}
\title{On Frobenius Betti numbers of graded rings of finite Cohen-Macaulay type}
\author{Nirmal Kotal}
\address{Chennai Mathematical Institute, Siruseri, Tamilnadu 603103. India}
\email{nirmal@cmi.ac.in}
\begin{abstract}
The notion of Frobenius Betti numbers generalizes the Hilbert-Kunz multiplicity theory and serves as an invariant that measures singularity.
However, the explicit computation of the Frobenius Betti numbers of rings has been limited to very few specific cases. This article focuses on the explicit computation of Frobenius Betti numbers of Cohen-Macaulay graded rings of finite Cohen-Macaulay type.
\end{abstract}

\maketitle
\section{Introduction}
\numberwithin{equation}{section}
Let $(R,\frakm,\bbk)$ denote a local ring of prime characteristic $p$ and of dimension $d$. Let $F:R\to R$ be the \emph{Frobenius map} that sends an element $r$ of $R$ to $r^p$. For $e\in \naturals$, the $R$-module $R$ via the $e$-fold iterated Frobenius map $F^e:R\to R$, is denoted as $\er$. We also assume $R$ is \emph{$F$-finite}, that is, $\er$ is a finitely generated $R$-module. When this is the case, the vector space dimension $\left[\bbk:\bbk^p \right]$ is finite, say $\alpha=\left[\bbk:\bbk^p \right]$. For an ideal $I\subset R$, define the $p^e$-th Frobenius power of $I$ as $I^{\left[p^e\right]} := \left(r^{p^e} \mid r \in  I\right)$. 

The \emph{Hilbert-Kunz multiplicity} of $R$ is
\begin{align*}
	\ehk(R):=&\lim_{e\to \infty}\frac{\lambda\left( R/\frakm^{[p^e]}\right)}{p^{de}},
\end{align*}
where $\lambda\left(\cdot\right)$ denotes the length function. The Hilbert-Kunz multiplicity provides a quantitative measure of singularity. For instance, when $R$ is unmixed, then $\ehk(R)=1$ if and only if $R$ is regular \cite[Theorem 1.5]{WY00hilbertkunz}. Furthermore, under certain mild conditions, the ring becomes Gorenstein and $F$-regular when the Hilbert-Kunz multiplicity approaches $1$ \cite[Corollary~3.6]{AE08lower}.

Let $a_e$ be the highest rank of a free $R$-module that appears in a direct sum decomposition of $\er$. Write $\er=R^{a_e}\oplus M_e$ as an $R$-module, where $M_e$ has no free direct summands. The \emph{$F$-signature} of $R$, represented as $s(R)$, is
\begin{align*}
	s(R):=&\lim_{e\to \infty}\frac{a_e}{p^{(\alpha+d)e}}.
\end{align*}
The $F$-signature is capable of identifying singularities as well. For example, when $R$ is reduced, then the $F$-signature is positive if and only if $R$ is strongly $F$-regular \cite[Theorem 0.2]{AL03the_fsignature}.

The $i$-th \emph{Betti number} of $\er$ is $\beta_i(\er)= \dim_\bbk \Tor_{i}(\er,\bbk)$. In other words, we have a free resolution of $\er$ as follows:
\begin{align*}
	\cdots  \to R^{\beta_{2}(\er)} \to R^{\beta_{1}(\er)}\to R^{\beta_{0}(\er)} \to \!\er \to 0.
\end{align*}
Define the $i$-th \emph{Frobenius Betti number} as
\begin{align*}
	\beta_i^F(R):=&\lim_{e\to \infty}\frac{\beta_i(\er)}{p^{(\alpha+d)e}}.
\end{align*}
In particular, $\beta_{0}^F(R)$ is the Hilbert-Kunz multiplicity of $R$. The notion of Frobenius Betti numbers was first introduced in \cite{Li08char} and subsequently further extended in \cite{SH17frobenius_betti}. These invariants can also be used to detect regularity. Specifically, $R$ is regular if and only if $\beta_i^F(R)=0$ for some $i>0$ \cite[Corollary 3.2]{AL08asymptotic}. If $R$ is a hypersurface, then $\beta_i^F(R)=\ehk(R)-s(R)$ for all $i>0$ \cite[Example 3.2]{SH17frobenius_betti}.

Computing all the Frobenius Betti numbers is as challenging as computing the syzygies. Even if $R$ has \emph{finite $F$-representation type} (in short, \emph{FFRT}), meaning that the collection of isomorphism classes of indecomposable modules that appear as a direct summand of $\er$ as an $R$-module for some $e$ is finite, there is almost no control over the syzygies of $\er$. Another nice class of rings would be \emph{finite Cohen-Macaulay type} (in short, finite CM type), that is, up to isomorphism, $R$ has only a finite number of indecomposable maximal Cohen-Macaulay (MCM) $R$-modules. When $R$ is a  Cohen-Macaulay (CM) ring, the class of finite CM type rings is a subclass of FFRT rings.

Eisenbud and Herzog classified graded CM rings of finite CM type in characteristic zero setup \cite{EisenbudHerzogFinCM1988}. However, the proof can be extended to prime characteristic setup (see Section \ref{sec:classification_cm_finite_type}). More precisely, we show the following:

\begin{proposition}
	\label{prop:finite_CM_type}
	Let $\Bbbk$ be an algebraically closed field of characteristic $p \neq 2$ and $R$ be a standard graded CM $\Bbbk$-algebra of finite CM type of dimension $d\geq 2$. Suppose that $R$ is neither regular, nor a hypersurface. Then $R$ is isomorphic to one of the following rings:
	\begin{enumerate}[(i)]
		\item Scroll of type $m, m\geq 3$, that is,
		$\frac{\bbk[x_0,\ldots,x_m]}{I}$, where $I$ is the ideal generated by all the $2\times 2$ minor of the matrix $\begin{pmatrix}
			x_0 & x_1 & \cdots &x_{m-1}\\
			x_1 & x_2 &\cdots & x_{m}
		\end{pmatrix}.$ These rings are sometimes referred to as the rational normal curve of degree $m$, and are isomorphic to the $\bbk$-algebra $\bbk[x^m,x^{m-1}y,\ldots,y^m]$.
		\item Scroll of type $(2,1)$, that is, $\frac{\bbk[x_0,x_1,x_2,y_0,y_1]}{I}$, where $I$ is the ideal generated by all the $2\times2$ minor of the matrix $\begin{pmatrix}
			x_0 & x_1  &y_0\\
			x_1 & x_2 & y_1
		\end{pmatrix}$, which is same as the $\bbk$-algebra $\mathbb{k}[x^2,xy,y^2,xz,yz]$.
		\item  The invariant ring $\Bbbk[x_1,x_2,x_3]^{\ints/2\ints}$, which is  same as the $\bbk$-algebra  $\bbk[x^2,y^2,z^2,xy,xz,yz]$.
	\end{enumerate}
\end{proposition}

As previously stated, the Frobenius Betti numbers are zero for regular rings and remain constant for hypersurfaces; we are specifically interested in the non-regular and non-hypersurface cases.
We explicitly compute the Frobenius Betti numbers (in prime characteristic setup) for $(1)$ $\bbk[x^\delta,x^{\delta-1}y,\ldots,y^\delta]$, $\delta \geq 2$ (Section \ref{sec:scroll_m}), $(2)$ $\mathbb{k}[x^2,xy,y^2,xz,yz]$ (Section \ref{sec:scroll_21}), and $(3)$ $\bbk[x^2,y^2,z^2,xy,xz,yz], \charact \bbk \neq 2$ (Section \ref{sec:fixed_ring}). Thus, the computation encompasses all the standard graded CM $\bbk$-algebra of finite CM type. While we also obtain the Hilbert-Kunz multiplicity and the $F$-signature from the computation, we remark that they were already known in some cases.
 
We summarize the main results in Table \ref{table1}.
\begin{center}
	\makegapedcells
	\begin{table}[h!]
		\label{table1}
	
\begin{tabular}{|c|c|c|c|c|}	
	\hline 
	$R$&
	\begin{tabular}{c}$s(R)$
	\end{tabular}
&\begin{tabular}{c}
		 $\ehk(R)$ 
	\end{tabular} &
\begin{tabular}{c}$\beta_{i}^F(R),$\\ $i\geq 1$
\end{tabular}& Reference
\\
\hline
\begin{tabular}{c}
	$\mathbb{k}[x^\delta,x^{\delta-1}y,\ldots,y^\delta],$\\ $\delta\geq 2$
\end{tabular}
&$\frac{1}{\delta}$ &$\frac{\delta+1}{2}$&$\frac{\delta(\delta-1)^i}{2}$&Theorem \ref{thm:rational_normal_curve}\\
\hline
$\bbk[x^2,xy,y^2,xz,yz]$ & $\frac{5}{12}$ & $\frac{7}{4}$ & $\frac{9}{4}\times 2^{i-1}$&Theorem \ref{thm:scroll21}\\
\hline 
\begin{tabular}{c}
	$\bbk [x^2,y^2,z^2,xy,xz,yz],$ \\ $\operatorname{char}\bbk\neq 2$
\end{tabular} & $\frac{1}{2}$ & $2$ & $4\times 3^{i-1}$&Theorem \ref{thm:second_veronese}\\
\hline 
\end{tabular}
\caption{Frobenius Betti numbers computation}
	\end{table}
\end{center}

\section{Preliminaries}
\numberwithin{equation}{section}
Throughout, all rings are Noetherian. Let $(R,\frakm,\bbk)$ be a local ring and $M$ be a finitely generated $R$-module. We denote $\mu(M)$ as the cardinality of a minimal generating set of $M$ as an $R$-module. Thus, $\mu(M)=\dim_{\bbk}\left(\frac{M}{\frakm M}\right)=\dim_{\bbk}\left( M\otimes_R \bbk \right)$.

Let $\bbk$ be a field, and $R=\oplus_{i\geq 0}R_i$ be a finitely generated standard graded $\bbk$-algebra, that is, $R_0=\bbk$ and $R$ is generated as a $\bbk$-algebra by homogeneous elements of degree $1$. Let $M=\oplus_iM_i$ be a finitely generated graded $R$-module. Define the \emph{Hilbert series} as $H(M,t):=\sum_i \dim_{\bbk}M_it^i$. For an integer $l$, define $M(l)$ to be the same as $M$ but with a change in grades as follows: the $i$-th graded piece of $M(l)$ is the $l+i$-th graded piece of $M$. Note that $H(M,t)=t^lH(M(l),t)$ for all $l$.
\subsection{Prime characteristic notions}
Let $(R,\frakm,\bbk)$ be an $F$-finite local domain of prime characteristic $p$ and $\alpha=\left[\bbk:\bbk^p \right]$. Let $K$ be the quotient field of $R$, and $\overline{K}$ denotes its algebraic closure. Thus, we have the inclusion $R\subset K \subset \overline{K}$. Define $R^{\frac{1}{p^e}}:=\left\{x\in \overline{K} \mid x^{p^e}\in R  \right\}$. That is, $R^{\frac{1}{p^e}}$ is the collection of all $p^e$-th roots of elements of $R$. In this setup, we can view the Frobenius map as the natural inclusion $R\hookrightarrow R^{\frac{1}{p^e}}$, implying that $\er\simeq R^{\frac{1}{p^e}}$ as an $R$-module.

The following observation is a consequence of the arguments in the proof of \cite[Proposition 3.10]{
sannai15}. Assume that $R$ is Cohen-Macaulay but not Gorenstein. Suppose that $\er=R^{a_e}\oplus \omega_R^{b_e}\oplus M_e$; where $a_e,b_e$ are some integers, $\omega_R$ is the canonical module, and $M_e$ has no direct summand of $R$ or $\omega_R$. Then 
\begin{align}
	\label{eq:sr_ae_be}
	s(R)=\lim_{e\to \infty} \frac{a_e}{p^{(\alpha+d)e}}=\lim_{e\to \infty} \frac{b_e}{p^{(\alpha+d)e}}.
\end{align}
With the view of the following proposition to compute the Frobenius Betti numbers (which also include Hilbert-Kunz multiplicity), if necessary, we may assume $R$ is complete and $\bbk$ is an algebraically closed field.
\begin{proposition}{\cite[Proposition 3.14]{SH17frobenius_betti}}
	Let $(R,\frakm,\bbk)\to (S,\mathfrak{n},\bbk')$ be a flat map between two $F$-finite local rings such that $\frakm S=\mathfrak{n}$. Then $\beta_{i,R}^F(R)=\beta_{i,S}^F(S)$ for all $i\geq 0$. In particular, $\beta_{i,R}^F(R)=\beta_{i,\widehat{R}}^F(\widehat{R})$ for all  $i\geq 0$, where $\widehat{R}$ denotes the completion of $R$.
\end{proposition}
\subsection{Finite CM type}
Let $(R,\frakm,\bbk)$ be either a local ring or a standard graded $\bbk$-algebra. A nonzero finitely generated $R$-module $M$ is called \emph{maximal Cohen-Macaulay module} (in short, MCM) if there exists a sequence of elements $x_1,\ldots,x_d$ that are $M$-regular and belong to the maximal ideal $\frakm$, where $d=\dim R$. The ring $R$ is said to have \emph{finite Cohen-Macaulay type} (in short, finite CM type), if there exists, up to isomorphism, only a finite number of indecomposable maximal Cohen-Macaulay (MCM) $R$-modules.

Here, we note two results that will be referred to multiple times.
\begin{theorem} Let $(R,\frakm,\bbk)$ be either a local ring or a standard graded $\bbk$-algebra.
	\begin{enumerate}
		\item \cite[Theorem 5]{AR89_CohenMacaulay_modules}, \cite[Main Theorem]{LWascent_of_finite} $R$ has finite CM type if and only if its completion $\widehat{R}$ has finite CM type.
		\item \cite[Theorem 10.1]{LeuschkeWiegandCMReprs2012} Assume $R$ be a $\bbk$-algebra and $\bbk\subset \bbk'$ is a field extension. If $R\otimes_\bbk \bbk'$ has finite CM type, then so has $R$.	
	\end{enumerate}
\end{theorem}

\subsection{Pick's theorem}
\label{thm:pick}
Consider a polygon in the plane with integer coordinates for each of its vertex. Let $n$ denote the number of integer points on the polygon's border, including vertices and sides, and $A$ denote its area. Then, the number of integer points within the polygon (including the boundary) is exactly $A+\frac{n}{2}+1$. This is known as Pick's theorem \cite{NZ67pick}.

\section{Cohen-Macaulay graded rings of finite CM type in prime characteristic}
\label{sec:classification_cm_finite_type}
\numberwithin{equation}{section}
In \cite{EisenbudHerzogFinCM1988} and \cite{YoshinoCMModules1990,LeuschkeWiegandCMReprs2012}, the
characterization of Cohen-Macaulay graded rings of finite CM type
is stated for fields of characteristic $0$. We want to observe below that
essentially those arguments will give the same characterization in prime
characteristic $p>2$. For the sake of completeness, we outline how the
proof proceeds in characteristic $0$ and make the necessary changes.

Let $\Bbbk$ be an algebraically closed field of characteristic $p \neq 2$.
By $R$ we mean a standard graded CM $\Bbbk$-algebra of finite CM type of dimension $d\geq 2$.

\begin{enumerate}
		\item
		$R$ has an isolated singularity
		(Auslander)~\cite[4.22]{YoshinoCMModules1990} or
		\cite[7.12]{LeuschkeWiegandCMReprs2012}.
		In particular $R$ is normal, so it is a domain.
		This proof does not require that $\charact \Bbbk = 0$.
		\item
		$R$ is stretched~\cite[Theorem~A]{EisenbudHerzogFinCM1988}, that is, the Hilbert series is of the form
		\begin{align}
			\label{eq:hilbert_series}
			\frac{1+ct+t^2 + \cdots + t^r}{(1-t)^d},
		\end{align}
		where $c = \rank_\Bbbk R_1 - \dim R$, and $r$ is the Castelnuovo-Mumford regularity of $R$.
		
		The proof in~\cite[Theorem~A]{EisenbudHerzogFinCM1988} goes through with the assumption that $\Bbbk$ is infinite. There is no need to make any assumptions on the characteristic of $\Bbbk$.
		\item If $R$ is Gorenstein, then the completion $\widehat{R}$ is either a power series ring or a hypersurface (Herzog) \cite[Theorem 9.15]{LeuschkeWiegandCMReprs2012}.
		\item If $R$ is not Gorenstein, then its Hilbert series is $\frac{1+ct}{(1-t)^d}$. In particular, $R$ has minimum multiplicity. This is~\cite[Theorem~B]{EisenbudHerzogFinCM1988}, whose proof uses a `general position' argument from~\cite{ACGH85} in characteristic $0$. But one can avoid this as in~\cite[17.2]{YoshinoCMModules1990}. 
			
			More precisely, use a theorem of Stanley~\cite[Theorem 4.4.9]{BH93} to the numerator of the Hilbert series of $R$ in (\ref{eq:hilbert_series}) to deduce that $c=1$ or $r
			\leq 2$. If $c=1$ or $r = 2$, the numerator of the Hilbert series
			is symmetric, so by another theorem of Stanley~\cite[Corollary 4.4.6]{BH93}, $R$
			is Gorenstein. Hence $c > 1$ and $r = 1$.
			
			Thus, we conclude that $\Proj R$ is a smooth variety of minimum degree, and
			not a hypersurface, with no assumption on $\charact \Bbbk$.
			These are classified (in all characteristics)
			by~\cite{EisenbudHarrisMinDeg87}:
			scrolls and the Veronese surface $\projective^2_\Bbbk \hookrightarrow
			\projective^5_\Bbbk$.
			
			\item
			Of the scrolls, the only ones with finite CM type are scroll of types $(m)$, $(1,1 )$, and
			$(2,1)$, since $\Bbbk$ is
			infinite~\cite[Theorem~3.2,~p.~9]{AR89cmtypeofcmrings}. Of these, scroll of type $(1,1)$ is the quadric hypersurface.
			If $\charact \Bbbk  \neq 2$, then 
			the coordinate ring of the Veronese surface 
			$\projective^2_\Bbbk \hookrightarrow \projective^5_\Bbbk$
			is isomorphic to is the invariant ring $\Bbbk[x_1,x_2,x_3]^{\ints/2\ints}$, and it is of finite CM
			type~\cite[Theorem~4.1,~p.~17]{AR89cmtypeofcmrings}.
		\end{enumerate}
		
The preceding discussions complete the proof of Proposition \ref{prop:finite_CM_type}.

\section{$R=\bbk\!\left[x^\delta,x^{\delta-1}y,\ldots,y^\delta\right],\delta \geq 2$}
\label{sec:scroll_m}
\numberwithin{equation}{section}
Consider a polynomial ring $S=\mathbb{k}[x,y]$ over a field $\mathbb{k}$. Let $R$ be the homogeneous coordinate ring of the rational normal curve of degree $\delta \geq 2$, that is,
$R=\bigoplus_{n\geq 0}S_{n\delta}$, which is same as the $\bbk$-algebra $\bbk\!\left[x^\delta,x^{\delta-1}y,\ldots,y^\delta\right]$.

We state some known facts (see \cite[Section 1]{KS15cone_of_betti}, for instance).
The indecomposable maximal Cohen-Macaulay modules are exactly of the form:
\begin{align*}
	M^{(l)}=\bigoplus_{n\geq 0}S_{n\delta+l},
\end{align*}
for  $l=0,\ldots,\delta-1$. Fix an $l$ such that $1\leq l \leq \delta-1$. Note that $M^{(l)}=R\left<x^l,x^{l-1}y,\ldots,y^l\right>$, and $\left(M^{(\delta-1)}\right)^l$ is the first syzygy of $M^{(l)}$, that is,
\begin{align}
	\label{ex:mdel}
	0 \too \left(M^{(\delta-1)}\right)^l \too R^{l+1} \too M^{(l)} \too 0.
\end{align}
\begin{lemma}
	Assume the setup of this section. Then the $i$-th Betti number of $M^{(l)}$ for each $1\leq l \leq \delta-1$ is
	\begin{align*}
		\beta_{i}(M^{(l)})&=\begin{cases}
			l+1& \text{if } i=0\\
			\delta(\delta-1)^{i-1}l& \text{else}
		\end{cases}.
	\end{align*}
\end{lemma}
\begin{proof}
The short exact sequence in (\ref{ex:mdel}) gives an exact sequence
\begin{align*}
0\to \Tor_1(M^{(l)},\bbk)\to \left(M^{(\delta-1)}\right)^l\otimes \bbk \to \bbk^{l+1}\to M^{(l)}\otimes \bbk \to 0,\\
\text{and }\Tor_{i+1}\left(M^{(l)}, \bbk \right)= \Tor_{i}\left(\left(M^{(\delta-1)}\right)^l,\bbk \right) \text{ for all } i\geq 1.
\end{align*}
Since $\mu(M^{(l)})=l+1$, so $\dim_{\bbk}\Tor_1(M^{(l)},\bbk)=\dim_{\bbk} \left(M^{(\delta-1)}\right)^l\otimes \bbk$. Thus, $\dim_{\bbk}\Tor_{i+1}(M^{(l)},\bbk)=\dim_{\bbk}\Tor_i \left((M^{(\delta-1)})^l, \bbk\right)$ for all $i\geq 0$. That is,
\begin{align*}
	\beta_{i+1}(M^{(l)})=l\beta_i(M^{(\delta-1)}) \text{ for all }i\geq 0.
\end{align*}
In particular, $\beta_{i+1}(M^{(\delta-1)})=(\delta-1)\beta_i(M^{(\delta-1)})$ for all $i\geq 0$, and consequently $\beta_i(M^{(\delta-1)})=\delta(\delta-1)^{i}$ for all $i\geq 0$. Hence the result.
\end{proof}
\begin{discussion}
Let $e_1,\ldots,e_{l+1}$ be the natural basis of $R^{l+1}$ which map to $x^l,x^{l-1}y,\ldots,y^l$ of $M^{(l)}$ respectively in the exact sequence (\ref{ex:mdel}). More precisely, the graded setup of (\ref{ex:mdel}) is as follows:
\begin{align}
	\label{ex:mdel_graded}
	0 \too \left(M^{(\delta-1)}(-l-1)\right)^l \too R(-l)^{l+1} \too M^{(l)} \too 0.
\end{align}	
The Hilbert series of $M^{(l)}$ is $H(M^{(l)},t)=\sum_{k\geq 0} (k\delta+l+1)t^{k\delta+l}$. In particular, the Hilbert series of $M^{(\delta-1)}$ is $H(M^{(\delta-1)},t)=\sum_{k\geq 0} (k+1)\delta t^{(k+1)\delta-1}$. Thus, from (\ref{ex:mdel_graded}) the Hilbert series of $\operatorname{Syz}^1_R\left(M^{(l)}\right)$ is $$H\left(\operatorname{Syz}^1_R\left(M^{(l)}\right),t\right)=lt^{l+1}H\left(M^{(\delta-1)},t \right)=\sum_{k\geq 0} l(k+1)\delta t^{(k+1)\delta+l}.$$
In particular, $\operatorname{Syz}^1_R\left(M^{(l)}\right)$ is minimally generated by $\delta l$ elements.

Also, note that $x^{\delta-k}y^ke_m-x^{\delta-k+1}y^{k-1}e_{m+1}\in \operatorname{Syz}^1_R\left(M^{(l)}\right)$ for all $1\leq k\leq \delta$ and $1\leq m\leq l$. The set $\left\{ x^{\delta-k}y^ke_m-x^{\delta-k+1}y^{k-1}e_{m+1} \mid 1\leq k\leq \delta \text{ and }1\leq m\leq l \right\}$ is $\bbk$-linearly independent, as its elements are distinct and lives in the same degree in $R^{l+1}$. The $R$-module generated by the aforementioned set is minimally generated by $\delta l$ elements and forms a graded submodule of $\operatorname{Syz}^1_R\left(M^{(l)}\right)$, with the same Hilbert series. Hence,
$$\operatorname{Syz}^1_R\left(M^{(l)}\right)=R\left< \left\{ x^{\delta-k}y^ke_m-x^{\delta-k+1}y^{k-1}e_{m+1} \mid 1\leq k\leq \delta \text{ and }1\leq m\leq l \right\} \right>.$$
\end{discussion}
\begin{theorem}\label{thm:rational_normal_curve}
		Assume the setup of this section and $\bbk$ is an algebraically closed field of prime characteristic $p$. Then the $F$-signature, Hilbert-Kunz multiplicity, and the Frobenius Betti numbers of $R$ are as follows:
	\begin{eqnarray*}
		s(R)=\dfrac{1}{\delta},&\ehk(R)=\dfrac{\delta+1}{2},&\text{and } \beta_{i}^F(R)=\dfrac{\delta(\delta-1)^i}{2} \text{ for all } i\geq 1.
	\end{eqnarray*}
\end{theorem}
\begin{proof}
Note that a monomial $x^iy^j\in R$ if and only if $i+j\in \delta \mathbb{Z}$. Let $q=p^e>\delta$. Then as an $R$-module, $R^{\frac{1}{q}}$ is generated by
\begin{align*}
	\left\{ x^{\frac{i}{q}}y^{\frac{j}{q}} \mid 0\leq i,j< \delta q,  i+j\in \delta\mathbb{Z}, \text{ and whenever } (r-1)q\leq i<rq, \text{ then } 0 \leq j<(\delta-r+1)q  \right\}.
\end{align*}
For each $0\leq l\leq \delta-1$, consider the set of tuples in $\mathbb{Z}^2$:
\begin{align*}
	\calP(l)&:= \left\{ (i,j)\mid lq\leq i<(l+1)q, 0\leq j<q,\text{ and } i+j\in \delta\mathbb{Z} \right\}.
\end{align*}
For each $(i,j)\in \calP(l)$, define an $R$-submodule $N^l_{i,j}$ of $R^{\frac{1}{q}}$ as follows:
\begin{align*}
	N^{l}_{i,j}&:=R\left< \left\{ x^{\frac{i}{q}}y^{\frac{j}{q}},x^{\frac{i-q}{q}}y^{\frac{j+q}{q}},\ldots,x^{\frac{i-lq}{q}}y^{\frac{j+lq}{q}}  \right\} \right>.
\end{align*}
We can consider $R^{\frac{1}{q}}$ as a vector space over the field $\bbk$. We can also assign  $(\mathbb{Z}/q,\mathbb{Z}/q)$ degree to it and treat it as the usual monomial grading. Then, as a $\bbk$-vector space,
\begin{align*}
	N^l_{i,j}&=\oplus_{k\geq 0} \bbk\left<\left\{  x^{\frac{i+k\delta q}{q}}y^{\frac{j}{q}},x^{\frac{i+(k\delta-1)q}{q}}y^{\frac{j+q}{q}},\ldots,x^{\frac{i-lq}{q}}y^{\frac{j+(k\delta +l)q}{q}} \right\}  \right>.
\end{align*}
Note that
\begin{equation}
	\label{eq:dim_k_N_lij}
	\dim_{\bbk}\left[N^l_{i,j}\right]_{\frac{i+j}{q}+k\delta}=k\delta +l +1=\dim_{\bbk}\left[M^{(l)}\right]_{k\delta+l}.
\end{equation}

Now consider the natural map $\pi:R^{l+1}\to N^l_{i,j}$ that maps a basis element $e_m$ to  $x^{\frac{i-(m-1)q}{q}}y^{\frac{j+(m-1)q}{q}},$ for all $1\leq m\leq l+1$. Then $\pi (x^{\delta-k}y^ke_{m}-x^{\delta-k+1}y^{k-1}e_{m+1})=0$ for all $1\leq k\leq \delta$ and $1\leq m \leq l$.
That is, $\pi(\operatorname{Syz}_R^1(M^{{(l)}}))=0$.
Therefore, \( \pi \) induces a surjection \( M^{(l)} \to N^l_{i,j} \) that preserves the \( \bbk \)-grading. Thus, by applying (\ref{eq:dim_k_N_lij}), we obtain \( M^{(l)} \simeq N^l_{i,j} \).

As a $\bbk$-vector space
\begin{align*}
	R^{\frac{1}{q}}=\oplus_{0\leq l\leq \delta -1}\left( \oplus_{(i,j)\in \calP(l)}\left(
	\oplus_{k\geq 0} \bbk\left<\left\{  x^{\frac{i+k\delta q}{q}}y^{\frac{j}{q}},x^{\frac{i+(k\delta-1)q}{q}}y^{\frac{j+q}{q}},\ldots,x^{\frac{i-lq}{q}}y^{\frac{j+(k\delta +l)q}{q}} \right\}  \right>\right)\right).
\end{align*}
That is, $R^{\frac{1}{q}}=\oplus_{0\leq l\leq \delta -1}\left( \oplus_{(i,j)\in \calP(l)}\left( N^l_{i,j} \right)\right)$ as a $\bbk$-vector space, and consequently also as an $R$-module. Let $a_l$ be the cardinality of the set $\calP(l)$. Then $R^{\frac{1}{q}}\simeq \oplus_{0\leq l \leq \delta-1} \left(M^{(l)} \right)^{a_l}$ as an $R$-module. Hence, $\beta_{i}(R^{\frac{1}{q}})=\sum_{l=0}^{\delta-1}a_l\beta_{i}(M^{(l)})$ for all $i\geq 0$.

\noindent It remains to compute $a_l$. Note that
\begin{align*}
	\left\{(i,j)\mid lq\leq i<(l+1)q, 0\leq j<q\right\}&=\sqcup_{r=0}^{\delta-1}\{(i,j)\mid lq\leq i<(l+1)q, 0\leq j<q,\\
	&\hspace{4cm} \text{ and } i+j\equiv r\!\mod \delta \},
\end{align*}
where $\sqcup$ stands for a disjoint union.
The left hand side has cardinality $q^2$. Thus,
\begin{align*}
\#\left\{(i,j)\mid lq\leq i<(l+1)q, 0\leq j<q,\text{ and } i+j\equiv 0\!\mod \delta \right\}&=\frac{p^{2e}}{\delta}+\mathcal{O}(p^{2e-1}).
\end{align*}
That is,
\begin{align*}
	a_l=\#\calP(l)&=\frac{p^{2e}}{\delta}+\mathcal{O}(p^{2e-1}).
\end{align*}
The number of free summands of $R^{\frac{1}{q}}$ is $a_0$. Thus, the $F$-signature is $\lim_{e\to \infty} \frac{a_0}{p^{2e}}=\frac{1}{\delta}$. The $i$-th Frobenius Betti number is
\begin{align*}
	\beta_i^F(R)&=\lim_{e\to \infty} \frac{\sum_{l=0}^{\delta-1} a_l\beta_{i}(M^{(l)})}{p^{2e}}\\
	&=\begin{cases}
		\lim_{e\to \infty} \frac{\sum_{l=0}^{\delta-1} a_l(l+1)}{p^{2e}} & \text{ if }i= 0\\
		\lim_{e\to \infty} \frac{\sum_{l=0}^{\delta-1} a_l\delta(\delta-1)^{i-1}l}{p^{2e}} & \text{ if }i\geq 1
	\end{cases}\\
	&=\begin{cases}
		\frac{1}{\delta} \sum_{l=0}^{\delta-1} l+1& \text{ if }i= 0\\
		\frac{1}{\delta} \sum_{l=0}^{\delta-1} \delta(\delta-1)^{i-1}l
		 & \text{ if }i\geq 1
	\end{cases}\\
		&=\begin{cases}
		\frac{\delta+1}{2}
		& \text{ if }i= 0\\
		\frac{\delta(\delta-1)^{i}}{2}
		& \text{ if }i\geq 1
	\end{cases}.
\end{align*}
\end{proof}

\section{$R= \mathbb{k}[x^2,xy,y^2,xz,yz]$}
\label{sec:scroll_21}
\numberwithin{equation}{section}
Let $R= \mathbb{k}[x^2,xy,y^2,xz,yz]$. Since the Frobenius Betti numbers do not change under completion and flat extension, we may assume $R= \mathbb{k}[\![x^2,xy,y^2,xz,yz]\!]$, and $\bbk$ is an algebraically closed field of prime characteristic $p$. The list of indecomposable MCM modules and the exact sequences stated below are explained in the proof of \cite[Theorem 2.1]{AR89cmtypeofcmrings}. Let $A =(x^2,xy)$, $B=(x^2,xy,xz)$, $C=(x^2,xy,y^2)$.
Let $D$ be the first syzygy of $B$, that is,
\begin{align}
	\label{eq:kr3b}
	0\to D\to R^3 \to B \to 0.
\end{align}
Note that
\begin{align*}
	D&=R\left<-xze_1+x^2e_3, -yze_1+xye_3,-xze_2+xye_3,-yze_2+y^2e_3,-xye_1+x^2e_2,-y^2e_1+xye_2 \right>,
\end{align*}
where $e_1,e_2,e_3$ are natural basis elements of $R^3$ that map to $x^2,xy,xz$, respectively. The set of indecomposable maximal Cohen-Macaulay modules of $R$ is precisely $\{R,A,B,C,D \}$. Note that $A$ is the canonical module of $R$, and $B$ is the first syzygy of $A$.
\begin{align}
	\label{eq:br2a}
	0\to B \to R^2 \to A \to 0.
\end{align}
Dualizing the above
\begin{align}
	\label{eq:ra2c}
	0\to R \to A^2 \to C \to 0.
\end{align}
We also have an exact sequence
\begin{align}
	\label{eq:ra2bk}
	0\to R \to A^2\oplus B \to D \to 0.
\end{align}

\begin{lemma}
	Assume the setup of this section. Then we have the following Betti numbers for the MCM modules:
\begin{center}
	\begin{tabular}{cccc}
		$\beta_{i}(A)=\begin{cases}
			2 &\text{if } i=0,\\
			3\times 2^{i-1} &\text{for all } i\geq 1
		\end{cases},$&and
		&$\frac{\beta_i(D)}{2}=
		\beta_i(B)=\beta_i(C)=3\times 2^i$ for all $i\geq 0$.
	\end{tabular}
\end{center}

\end{lemma}
\begin{proof}
The short exact sequence in (\ref{eq:kr3b}) gives an exact sequence $0\to \Tor_1(B,\bbk)\to D\otimes\bbk \to \bbk^3 \to B\otimes \bbk \to 0$ and $\Tor_{i+1}(B,\bbk)=\Tor_i(D,\bbk)$ for all $i\geq 1$. Since $\mu(B)=3$, so $\dim_\bbk \Tor_1(B,\mathbb{k})=\dim_\bbk D\otimes\mathbb{k}$. Thus, $\dim_\bbk \Tor_{i+1}(B,\mathbb{k})=\dim_\bbk  \Tor _{i}(D,\mathbb{k})$ for all $i\geq 0$. That is,
\begin{align}
	\label{eq:beta_iKB}
	\beta_{i+1}(B)&=\beta_{i}(D) \quad \text{for all } i\geq 0.
\end{align}
In the same way, the short exact sequence in (\ref{eq:br2a}) yields
\begin{align}
	\label{eq:beta_iBA}
	\beta_{i+1}(A)&=\beta_{i}(B) \quad \text{for all } i \geq 0.
\end{align}
The short exact sequence in (\ref{eq:ra2c}) gives an exact sequence $0\to \Tor_1(A^2,\bbk) \to \Tor_1(C,\bbk)\to \bbk \to A^2\otimes\bbk \to  C\otimes \bbk \to 0$ and $\Tor_{i}(A^2,\bbk)=\Tor_i(C,\bbk)$ for all $i\geq 2$. Since $\mu(A^2)=4$ and $\mu(C)=3$, so $\dim_\bbk \Tor_1(A^2,\bbk)=\dim_\bbk \Tor_1(C,\bbk)$. Thus, $\dim_\bbk \Tor_i(A^2,\bbk)=\dim_\bbk \Tor_i(C,\bbk)$ for all $i\geq 1$. That is,
\begin{align}
	\label{eq:2beta_iAC}
	2\beta_{i}(A)&=\beta_{i}(C) \quad \text{for all } i\geq 1.
\end{align}
In the same way, the short exact sequence in (\ref{eq:ra2bk}) yields
\begin{align}
	\label{eq:2beta_iABK}
	2\beta_{i}(A)+\beta_{i}(B)&=\beta_{i}(D) \quad \text{for all } i \geq 1.
\end{align}
By combining (\ref{eq:beta_iKB}),(\ref{eq:beta_iBA}), and (\ref{eq:2beta_iABK}) together, we obtain
 $\beta_{i}(D)=2\beta_{i-2}(D)+\beta_{i-1}(D)$ for all $i\geq 2$. That is, $2\left( \beta_{i}(D)+\beta_{i+1}(D)\right)=\beta_{i+1}(D)+\beta_{i+2}(D)$ for all $i\geq 0$. Consequently $\beta_{i+1}(D)+\beta_{i+2}(D)=2^{i+1}\left(\beta_{0}(D)+\beta_1(D) \right)$ for all $i\geq 0$.
 
 \noindent
 Now $\beta_{0}(K)=6$, $\beta_1(A)=\beta_{0}(B)=3$, and $\beta_1(B)=\beta_{0}(D)=6$. Hence, $\beta_1(D)=2\beta_{1}(A)+\beta_1(B)=12$. Consequently, $\beta_{i+2}(D)=2^{i+1}\times 18 - \beta_{i+1}(D)$ for all $i\geq 0$. Hence, $\beta_i(D)=3\times 2^{i+1}$ for all $i\geq 0$. Therefore, $\beta_i(A)=3\times 2^{i-1}$ for all $i\geq 1$ and $\beta_i(B)=\beta_i(C)=3\times2^i$ for all $i\geq 0$.
\end{proof}
\begin{theorem}
	\label{thm:scroll21}
	Assume the setup of this section. Then the $F$-signature, Hilbert-Kunz multiplicity, and the Frobenius Betti numbers of $R$ are as follows:
\begin{eqnarray*}
		s(R)=\dfrac{5}{12}, & \ehk(R)= \dfrac{7}{4}, & \text{ and } \beta_i^F(R)= \dfrac{9}{4}\times 2^{i-1} \text{ for all } i\geq 1.
\end{eqnarray*}
	
\end{theorem}
\begin{proof}
Let $R'= \mathbb{k}[x^2,xy,y^2,xz,yz]$. A monomial $x^iy^jz^k$ belongs to $R'$ if and only if $i+j\geq k$ and $i+j+k$ is even.
 Let $q=p^e>2$. Consider the following set of triplets in $\mathbb{Z}^3$:
 \begin{align*}
 	\calP&:=\left\{ (i,j,k) \mid 0\leq i,j,k< 2q;  i+j+k\in 2\mathbb{Z}; i+j\geq k; \text{ at most one of } i,j,k \text{ can exceed } q \right\}.
 \end{align*}
Then $R'^{\frac{1}{q}}$ is generated by 
$\left\{ x^{\frac{i}{q}}y^{\frac{j}{q}} z^{\frac{k}{q}} \mid (i,j,k)\in \calP \right\}$. Divide the set $\calP$ into three disjoint subsets as follows:
\begin{align*}
	\calP(1)&:=\left\{(i,j,k)\mid 0\leq i,j,k< q,  i+j+k\in 2\mathbb{Z}, 0\leq i+j-k \right\},\\
	\calP(2)&:=\left\{(i,j,k)\mid 0\leq j,k< q, q\leq i<2q,  i+j+k\in 2\mathbb{Z}, 0\leq i+j-k< 2q \right\},\\
	\text{and }	\calP(3)&:=\left\{(i,j,k)\mid 0\leq j,k< q, q\leq i<2q,  i+j+k\in 2\mathbb{Z}, 2q\leq i+j-k \right\}.
\end{align*}
For each $(i,j,k)\in \calP$, define an $R$-submodule $N_{i,j,k}$ of $R'^{\frac{1}{q}}$ as follows:
\begin{align*}
	N_{i,j,k}=\begin{cases}
		R'\left< x^{\frac{i}{q}}y^{\frac{j}{q}}z^{\frac{k}{q}} \right> &\text{if } (i,j,k)\in \calP(1)\\
		R'\left< x^{\frac{i}{q}}y^{\frac{j}{q}}z^{\frac{k}{q}},x^{\frac{i-q}{q}}y^{\frac{j+q}{q}}z^{\frac{k}{q}} \right> & \text{if } (i,j,k)\in \calP(2)\\
		R'\left< x^{\frac{i}{q}}y^{\frac{j}{q}}z^{\frac{k}{q}},x^{\frac{i-q}{q}}y^{\frac{j+q}{q}}z^{\frac{k}{q}}, x^{\frac{i-q}{q}}y^{\frac{j}{q}}z^{\frac{k+q}{q}} \right> &\text{if } (i,j,k)\in \calP(3)
	\end{cases}.
\end{align*}
Note that $R'^{\frac{1}{q}}=\sum_{(i,j,k)\in \calP} N_{i,j,k}$. We claim that 
\begin{align}
	\label{claim:dir_sum_n_ijk}
	R'^{\frac{1}{q}}=\oplus_{(i,j,k)\in \calP} N_{i,j,k}.
\end{align}
Assume the claim. Since the completion of $R'$ is $R$, the number of minimal generators in each summand of  $R^{\frac{1}{q}}$ is the same as of $R'^{\frac{1}{q}}$. Also the number of minimal generators of each summand of $R'^{\frac{1}{q}}$ in the above decomposition (\ref{claim:dir_sum_n_ijk}) is at most $3$, it follows that $D$ can not be a summand of $R^{\frac{1}{q}}$ as $\mu(D)=6$. Assume that, $R^{\frac{1}{q}}=R^{a_e}\oplus A^{b_e}\oplus M_e$ for some integers $a_e,b_e$; and $B$ or $C$ are the only summands of $M_e$. We need to determine $a_e$, which will be used to compute the $F$-signature.

\noindent
Note that for each $(i,j,k)\in \calP(2)$, the generators of $N_{i,j,k}$ have a relation $$xyx^{\frac{i}{q}}y^{\frac{j}{q}}z^{\frac{k}{q}}-x^2x^{\frac{i-q}{q}}y^{\frac{j+q}{q}}z^{\frac{k}{q}}=0.$$ Similarly, for each $(i,j,k)\in \calP(3)$, the generators of $N_{i,j,k}$ have relations
\begin{eqnarray*}
	xyx^{\frac{i}{q}}y^{\frac{j}{q}}z^{\frac{k}{q}}-x^2x^{\frac{i-q}{q}}y^{\frac{j+q}{q}}z^{\frac{k}{q}}=0 &\text{and}&xzx^{\frac{i}{q}}y^{\frac{j}{q}}z^{\frac{k}{q}}-x^2 x^{\frac{i-q}{q}}y^{\frac{j}{q}}z^{\frac{k+q}{q}}=0.
\end{eqnarray*}
Thus, for each $(i,j,k)\in \calP(1)$, $N_{i,j,k}$ are the only free summands of $R'^{\frac{1}{q}}$. Since the number of free summands does not change under completion, $a_e$ is exactly the same as the cardinality of $\calP(1)$. Note that
\begin{align*}
	 \left\{(i,j,k)\mid 0\leq i,j,k< q, i+j\geq k \right\}&=\sqcup_{r=0}^1  \{(i,j,k)\mid 0\leq i,j,k< q, i+j\geq k,\\
	 &\hspace{4cm} i+j+k \equiv r \mod 2 \}.
\end{align*}
The cardinality of the left hand side is:
\begin{align*}
	\#\left\{(i,j,k)\mid 0\leq i,j,k< q, i+j\geq k \right\}&=\sum_{k=0}^{q-1} \#\left\{(i,j)\mid 0\leq i,j< q, i+j\geq k \right\} \\
	&= \sum_{k=0}^{q-1} \left( (q-1)^2-\frac{1}{2}k^2+\frac{4q-k-4}{2}+1 \right)\\
	&=q^3 -\frac{1}{2}\sum_{k=1}^{q-1}(k^2+k)\\
	&=q^3-\frac{1}{2}\left(\frac{q(q-1)(2q-1)}{6}+\frac{q(q-1)}{2} \right)\\
	&=\frac{5}{6}q^3+\frac{1}{6}q,
\end{align*}
where the second equality follows from Pick's theorem (\ref{thm:pick}). Thus,
\begin{align*}
	a_e=\#\calP(1)&= \#\{(i,j,k)\mid 0\leq i,j,k< q, i+j\geq k, i+j+k \equiv 0 \mod 2 \} \\
	&=\frac{5}{12}p^{3e}+\mathcal{O}(p^{3e-1}).
\end{align*}
Note that $\dim R=3$. Thus, the $F$-signature is $s(R)=\lim_{e\to \infty} \frac{a_e}{p^{3e}}=\frac{5}{12}$. Also, (\ref{eq:sr_ae_be}) yields $\lim_{e\to \infty} \frac{b_e}{p^{3e}}=\frac{5}{12}$. Also note that, $\rank R^{\frac{1}{q}}=q^d$ and $\rank A=1$. By using the short exact sequences in (\ref{eq:br2a}) and (\ref{eq:ra2c}), we get $\rank B=\rank C= 1$. Let $M_e$ is summand of $c_e$ copies of $B$ or $C$. Then $\lim_{e\to \infty}\frac{a_e +b_e \rank A +c_e \rank B}{p^{3e}} =1$. Therefore $\lim_{e\to \infty} \frac{c_e}{p^{3e}}=\frac{2}{12}$.
 Hence
 \begin{align*}
 	\beta_{i}^F(R)&=\lim_{e\to \infty} \frac{a_e\beta_i(R)+b_e\beta_{i}(A)+c_e\beta_{i}(B)}{p^{3e}}\\ &=\begin{cases}
 		\lim_{e\to \infty} \frac{a_e+2b_e+3c_e}{p^{3e}} & \text{ if }i=0\\
 		\lim_{e\to \infty} \frac{3\times 2^{i-1}b_e+3\times 2^ic_e}{p^{3e}} &\text{ if }i\geq 1
 	\end{cases} \\
 	&=\begin{cases}
 		\frac{7}{4}&\text{ if }i=0\\
 		\frac{9}{4}\times 2^{i-1} &\text{ if }i\geq 1
 	\end{cases}.
 \end{align*}
It remains to prove the claim \ref{claim:dir_sum_n_ijk}. It is enough to show that 
$$N_{i,j,k}\cap \sum_{\substack{(i',j',k')\in \calP,\\ (i',j',k')\neq (i,j,k)}} N_{i',j',k'}=0.$$
But this can be proved as a $\bbk$-vector space. A benefit of seeing $N_{i,j,k}$ as a $\bbk$-vector space is that we may assign  $(\mathbb{Z}/q,\mathbb{Z}/q,\mathbb{Z}/q)$ degree  to it and treat it as a standard monomial grading. Thus, it is enough to show $N_{i,j,k}\cap N_{i',j',k'}=0$ whenever $(i,j,k)\neq (i',j',k')$. Let $\xi$ be an element in the minimal generating set of $N_{i,j,k}$ and $x^{i_1}y^{j_1}z^{k_1}$ be a monomial in $R'$. Then, comparing the non-integral part of the powers of $x,y,z$ in $x^{i_1}y^{j_1}z^{k_1}\xi$, we can conclude that $x^{i_1}y^{j_1}z^{k_1}\xi \notin N_{i',j',k'}$ for any $(i',j',k')\in \calP$ different from $(i,j,k)$. Thus, $N_{i,j,k}\cap N_{i',j',k'}=0$ whenever $(i,j,k)\neq (i',j',k')$. This completes the proof.
\end{proof}

\section{$R=\bbk[x^2,y^2,z^2,xy,xz,yz]$}
\label{sec:fixed_ring}
\numberwithin{equation}{section}
Let $R=\bbk[x^2,y^2,z^2,xy,xz,yz]$, and $\operatorname{char}\bbk=p>2$.  Since the Frobenius Betti numbers do not change under completion and flat extension, we may assume $R= \mathbb{k}[\![x^2,y^2,z^2,xy,xz,yz]\!]$, and $\bbk$ is an algebraically closed field. Note that $\dim R=3$. The list of indecomposable MCM modules and the exact sequences stated below are explained in the proof of \cite[Theorem 4.1]{AR89cmtypeofcmrings}. Only indecomposable MCM modules are exactly $\{R,A,B\}$, where $A$ is the canonical module of $R$ and $B$ is the first syzygy of $A$, that is,
\begin{align}
	\label{ex:br3a}
	0\to B \to R^3 \to A \to 0.
\end{align}
 Dualizing the above exact sequence:
 \begin{align}
 	\label{ex:ra3b}
 	0\to R \to A^3 \to B \to 0.
 \end{align}
 \begin{lemma}
 	Assume the setup of this section. Then we have the following Betti numbers for the MCM modules:
 \begin{center}
 \begin{tabular}{ccc}
 	$\beta_i(A)=\begin{cases}
 		3 &\text{if }i=0\\
 		8\times 3^{i-1} &\text{for all }i\geq 1
 	\end{cases}$,
 	&and& $\beta_i(B)=8\times 3^{i}$ for all $i\geq 0$. 
 \end{tabular}
\end{center}
 	
 \end{lemma}
\begin{proof}
From the short exact sequence in (\ref{ex:br3a}), we get an exact sequence $0\to \Tor_1(A,\mathbb{k})\to B\otimes\mathbb{k}\to \mathbb{k}^3\to A\otimes \mathbb{k} \to 0$ and $\Tor_{i}(B,\mathbb{k})=\Tor_{i+1}(A,\mathbb{k})$ for all $i\geq 1$. Since $\mu(A)=3$, so $\dim_\bbk \Tor_1(A,\mathbb{k})=\dim_\bbk B\otimes\mathbb{k}$. Thus, $\dim_\bbk \Tor_{i+1}(A,\mathbb{k})=\dim_\bbk \Tor_{i}( B,\mathbb{k})$ for all $i\geq 0$. That is,
\begin{align*}
	\beta_{i+1}(A)=\beta_i(B) \text{ for all }i\geq 0.
\end{align*}
From the short exact sequence in (\ref{ex:ra3b}), we get an exact sequence
 $0\to \Tor_1(A^3,\mathbb{k})\to \Tor_1(B,\mathbb{k}) \to  \mathbb{k} \to A^3\otimes \mathbb{k} \to B\otimes\mathbb{k} \to 0$ and $\Tor_{i}(A^3,\mathbb{k})=\Tor_{i}(B,\mathbb{k})$ for all $i\geq 2$. This gives
\begin{align}
\label{eq:3beta1a}	3\beta_1(A)-\beta_1(B)+1-3\beta_{0}(A)+\beta_{0}(B)=0,\\
	\text{and }3\beta_i(A)=\beta_i(B)\text{ for all } i\geq 2.
\end{align}
  The Hilbert series of $R$ is $H(R,t)=\frac{1+3t}{(1-t)^3}$. So the Hilbert series of $A$ is $H(A,t)=-H(R,t^{-1})=\frac{t^2(3+t)}{(1-t)^3}$ \cite[Corollary 4.4.6]{BH93}. This shows that $A$ is minimally generated by three elements, all of which are of degree $2$. Consequently,  (\ref{ex:br3a}) can be expressed in the graded setting as:
  \begin{equation}
  	\label{ex:br3a_graded}
  	0\to B \to \left(R(-2)\right)^3 \to A \to 0.
  \end{equation}
From the above graded exact sequence, the Hilbert series of $B$ is determined as follows: 
  \begin{align*}
  	H(B,t)&=H(\left(R(-2)\right)^3,t)-H(A,t)\\
  	&=3t^2H(R,t)-H(A,t)\\
  	&=\frac{8t^3}{(1-t)^3}.
  \end{align*}
Therefore, $B$ is minimally generated by $8$ elements. Hence $\beta_{0}(B)=8$ and so $\beta_1(A)=\beta_{0}(B)=8$. Also (\ref{eq:3beta1a}) gives $\beta_{1}(B)=24$. So $\beta_2(A)=\beta_{1}(B)=24$. Now $3\beta_i(A)=\beta_i(B)=\beta_{i+1}(A)$ for all $i\geq 1$, gives $\beta_i(A)=8\times 3^{i-1}$ for all $i\geq 1$. Consequently, $\beta_{i}(B)=\beta_{i+1}(A)=8\times 3^i$ for all $i\geq 0$.
\end{proof}
\begin{theorem}\label{thm:second_veronese}
	Assume the setup of this section. Then the $F$-signature, Hilbert-Kunz multiplicity, and the Frobenius Betti numbers of $R$ are as follows:
	\begin{eqnarray*}
		s(R)=\frac{1}{2},&\ehk(R)=2,&\text{ and  }\beta_i^F(R)=4\times 3^{i-1} \text{ for all }i\geq 1.
	\end{eqnarray*}
\end{theorem}
\begin{proof}
 Utilize \cite[Example 4.7, Theorem 4.6]{JNSWY23lower}, to observe that $R$ and $A$ are the only summands of  $R^{\frac{1}{p^e}}$, the Hilbert-Kunz multiplicity $e_{HK}(R)=2$, and the $F$-signature $s(R)=\frac{1}{2}$. Assume  $R^{\frac{1}{p^e}}=R^{a_e}\oplus A^{b_e}$ for some integers $a_e$ and $b_e$. Then $s(R)= \lim_{e\to \infty} \frac{a_e}{p^{3e}}=\frac{1}{2}$. Also from (\ref{eq:sr_ae_be}), we get $\lim_{e\to \infty} \frac{b_e}{p^{3e}}=\frac{1}{2}$. Now $\beta_i(R^{\frac{1}{p^e}})=b_e\beta_i(A)=8b_e\times 3^{i-1}$ for all $i\geq 1$. Hence, $\beta^{F}_i(R)=\lim_{e\to \infty} \frac{\beta
_i(R^{\frac{1}{p^e}})}{p^{3e}}= 4\times 3^{i-1}$ for all $i\geq 1$.
\end{proof}
\section{Acknowledgement}
The author thanks his advisor, Manoj Kummini, for all the valuable discussions and insightful suggestions; indeed, it is his observation that the proof of the classification of Cohen-Macaulay graded rings of finite CM type can be extended to prime characteristic setup, as explained in Section \ref{sec:classification_cm_finite_type}. The computer algebra system \texttt{Macaulay2} \cite{M2} proved to be quite beneficial in analyzing many examples.

A grant from the Infosys Foundation partially supported the author.

\end{document}